\documentclass[10pt,reqno]{amsart}
\usepackage{bbm}
\usepackage{mathrsfs}
\usepackage{amsfonts} 
\usepackage[dvipsnames,usenames]{color}
\usepackage{graphicx,latexsym,bm,amsmath,amssymb,verbatim,multicol,lscape}
\newtheorem{thm}{Theorem} [section]
\newtheorem{lem}{Lemma}[section]

\theoremstyle{definition}

\theoremstyle{remark}

\numberwithin{equation}{section}

\allowdisplaybreaks

\begin{document}
\title{On the sums of squares of exceptional units in residue class rings}
\author[Y.L. Feng]{Yulu Feng}
\address{Mathematical College, Sichuan University, Chengdu 610064, P.R. China}
\email{yulufeng17@126.com (Y.L. Feng)}
\author[S.F. Hong]{Shaofang Hong$^*$}
\address{Mathematical College, Sichuan University, Chengdu 610064, P.R. China}
\email{sfhong@scu.edu.cn; s-f.hong@tom.com; hongsf02@yahoo.com}
\begin{abstract}
Let $n\ge 1, e\ge 1, k\ge 2$ and $c$ be integers. An integer
$u$ is called a unit in the ring $\mathbb{Z}_n$ of residue
classes modulo $n$ if $\gcd(u, n)=1$. A unit $u$ is called
an exceptional unit in the ring $\mathbb{Z}_n$ if $\gcd(1-u,n)=1$.
We denote by $\mathcal{N}_{k,c,e}(n)$ the number of solutions
$(x_1,...,x_k)$ of the congruence $x_1^e+...+x_k^e\equiv c
\pmod n$ with all $x_i$ being exceptional units in the ring
$\mathbb{Z}_n$. In 2017, Mollahajiaghaei presented a formula
for the number of solutions $(x_1,...,x_k)$ of the congruence
$x_1^2+...+x_k^2\equiv c\pmod n$ with all $x_i$ being the units
in the ring $\mathbb{Z}_n$. Meanwhile, Yang and Zhao gave an exact
formula for $\mathcal{N}_{k,c,1}(n)$. In this paper, by using
Hensel's lemma, exponential sums and quadratic Gauss sums, we derive
an explicit formula for the number $\mathcal{N}_{k,c,2}(n)$.
Our result extends Mollahajiaghaei's theorem and that of Yang and Zhao.
\end{abstract}
\thanks{$^*$S.F. Hong is the corresponding author and was supported
partially by National Science Foundation of China Grants \#11771304, \#12171332.}
\maketitle

\section{Introduction}
Let $n$ be a positive integer and let $\mathbb Z_n=\{0,1,...,n-1\}$
be the ring of residue classes modulo $n$. Then the set of units
of $\mathbb Z_n$ is $\mathbb{Z}_n^*=\{u\in\mathbb{Z}_n|\gcd(u,n)=1\}$.
Let $e, k$ and $c$ be integers with $e\ge 1$ and $k\ge 2$. We define
$N_{k,c,e}(n)$ to be the number of solutions $(x_1, ..., x_k)\in
(\mathbb{Z}_n^*)^k$ of the following $e$-th diagonal congruence:
\begin{align}\label{1.1}
x_1^e+\cdots+x_k^e\equiv c\pmod n.
\end{align}
For any finite set $S$, we denote by $|S|$ the number of the
elements in $S$. Then
$$N_{k,c,e}(n)=|\{(x_1, ..., x_k)\in(\mathbb{Z}_n^*)^k|x_1^e+
\cdots+x_k^e\equiv c\pmod n\}|.$$
In 1925, Rademacher \cite{[R]} raised the problem
of computing $N_{k,c,1}(n)$. In 1926, Brauer
\cite{[B]} answered this problem. In 2009, by using the
multiplicativity of $N_{k,c,1}(n)$ with respect to $n$,
Sander \cite{[SJ1]} gave a new proof of this problem
for the case $k=2$. In 2015, Yang and Tang \cite{[YT]}
extended Sander's result to the quadratic case by giving
an explicit formula for $N_{2,c,2}(n)$. In 2017,
Mollahajiaghaei \cite{[Mo]} generalized Yang and Tang's
theorem by presenting a formula for $N_{k,c,2}(n)$.

On the other hand, let $R$ be a finite commutative ring with identity
$1_R$ and let $R^*$ be the multiplicative group of $R$. In 1969,
for the sake of solving certain cubic Diophantine equations,
Nagell \cite{[NT]} first introduced the concept of exceptional
units. A unit $u\in R^*$ is called an {\it exceptional unit}
if $1_R-u\in R^*$. We write $R^{**}$ for the set of all
exceptional units of $R$. The exceptional unit is very important
tool in studying lots of types of Diophantine equations including
Thue equations \cite{[TN1]}, Thue-Mahler equations \cite{[TN2]},
discriminant form equations \cite{[SN2]}. Lenstra \cite{[L]}
used exceptional units to find Euclidean number fields.
Since then, many new Euclidean number fields were found,
see, for example, \cite{[H], [NS]}. Besides, exceptional units
are connected with the investigation of cyclic resultants
\cite{[St],[St2]}, Salem numbers and Lehmer's conjecture
related to Mahler's measure \cite{[Si],[Si2]}.

For the case $R=\mathbb Z_n$, it is clear that
$\mathbb{Z}_n^{**}=\{u\in\mathbb{Z}_n|\gcd(u,n)=\gcd(1-u,n)=1\}$.
We define $\mathcal N_{k,c,e}(n)$ to be the number
of exceptional unit solutions $(x_1, ..., x_k)\in(\mathbb{Z}_n^{**})^k$
of the $e$-th diagonal congruence (\ref{1.1}), namely,
$$
\mathcal N_{k,c,e}(n)=|\{(x_1, ..., x_k)\in(\mathbb{Z}_n^{**})^k:
x_1^e+\cdots+x_k^e\equiv c\pmod n\}|.
$$
Obviously, $\mathcal N_{k,c,e}(1)=1$.
In 2016, Sander \cite{[SJ]} obtained a formula for
$\mathcal N_{2,c,1}(n)$. One year later, Yang and Zhao
\cite{[YZ]} generalized Sander's result by giving
the following explicit formula:
$$
\mathcal N_{k,c,1}(n)=(-1)^{k\omega(n)}\prod_{p| n}p^{k\nu_p(n)-\nu_p(n)-k}
\Big(p\sum_{j=0\atop{j\equiv c\pmod p}}^k\binom{k}{j}+(2-p)^k-2^k\Big),
$$
where $\omega(n):=\sum_{p \ {\rm prime}, \ p|n}1$ stands for
the number of distinct prime divisors of $n$. We point out that
an error in the formula in Theorem 1 of \cite{[YZ]} was corrected
by Zhao {\it et al.} \cite{[ZHZ]}, where the sign factor $(-1)^k$
should read $(-1)^{k\omega(n)}$.

In this paper, we address the problem of computing the
number of exceptional unit solutions $(x_1, ..., x_k)$ of
quadratic congruence $x_1^2+\cdots+x_k^2\equiv c\pmod n$.
We will make use of Hensel's lemma,
quadratic Gauss sums and exponential sums
to deduce an explicit formula of $\mathcal{N}_{k,c,2}(n)$
for any positive integer $n$. To state our main result,
we have to introduce some notation and concepts. We denote
$\mathbb{Z}, \mathbb{Z}_{\ge 0}$ and $\mathbb{Z}^+$ to
be the set of integers, the set of nonnegative integers
and the set of positive integers, respectively. For any
prime number $p$ and $x\in\mathbb{Z}$, we let
$e(\frac{x}{p}):=\exp(\frac{2{\pi}{\rm i}x}{p})$ and
denote by $\big(\frac{x}{p}\big)$ the Legendre symbol.
In addition, for any $y,z\in\mathbb{Z}_{\ge 0}$, we denote
by $\binom{y}{z}$ the binomial coefficient, that is,
$\binom{y}{z}=\frac{y!}{z!(y-z)!}$.
Notices that $\binom{0}{0}:=1$. Moreover, for any
$m\in\mathbb{Z}^+$, there exist unique integers $a$ and $r$
with $p\nmid a$ and $r\ge 0$, such that $m=ap^r$. The number $r$
is called the {\it$p$-adic valuation} of $m$, denoted by $r=v_p(m)$.
We can now state the main result of this paper.

\begin{thm}\label{thm1}
Let $n\ge 1, k\ge 2$ and $c$ be integers. If $n$ is an even integer,
then $\mathcal{N}_{k,c,2}(n)=0$. If $n$ is odd, then
\begin{align*}
&\mathcal{N}_{k,c,2}(n)\\
=&(-1)^{k\omega(n)}\prod_{p|n}p^{k\nu_p(n)-\nu_p(n)-k}
\Bigg((2-p)^k-\sum_{i=0}^{\lfloor\frac{k}{2}\rfloor}
(-1)^{\frac{(p-1)i}{2}}p^{i}\binom{k}{2i}\Big(2^{k-2i}
-p\sum_{j=0\atop{p|(j-c)}}^{k-2i}\binom{k-2i}{j}\Big)\notag\\
&-\sum_{i=0}^{\lfloor\frac{k-1}{2}\rfloor}
(-1)^{\frac{(p-1)(i+1)}{2}}p^{i+1}\binom{k}{2i+1}
\sum_{j=0\atop{p\nmid(j-c)}}^{k-2i-1}
\binom{k-2i-1}{j}\Big(\frac{j-c}{p}\Big)\Bigg).
\end{align*}
\end{thm}
Evidently, Theorem \ref{thm1} generalizes the Yang-Zhao
theorem \cite{[YZ]} from the liner case to the quadratic
case, and extends Mollahajiaghaei's theorem \cite{[Mo]}
from the unit solution case to the exceptional unit
solution case.

This paper is organized as follows. We present in Section
2 several lemmas that are needed in the proof of Theorem
\ref{thm1}. Finally, Section 3 is devoted to the
proof of Theorem \ref{thm1}.

\section{Preliminary lemmas}
In this section, we present some lemmas that we require to
prove Theorem \ref{thm1}. We begin with the well-known
Chinese remainder theorem.

\begin{lem} \label{lem1} \cite{[A]}
Let $r$ be a positive integer, $b_1,...,b_r$ be arbitrary integers
and $m_1,..., m_r$ be positive integers relatively prime in pairs.
Then the system of congruences
$$
x\equiv b_i\pmod{m_i}\ (i=1,...,r)
$$
has exactly one solution modulo the product $m_1\cdots m_r$.
\end{lem}

Consequently, we use the above lemma to show the multiplicativity
of the arithmetic function $\mathcal{N}_{k,c,e}$.

\begin{lem}\label{lem2}
Let $k, c$ and $e$ be integers with $k\ge 2$ and $e\ge 1$.
Then $\mathcal{N}_{k,c,e}$ is a multiplicative function.
\end{lem}

\begin{proof}
First of all, for positive inter $n$, we define the set $S_{k,c,e}(n)$ as follows:
\begin{align}\label{2.1}
S_{k,c,e}(n):=\{(x_1,...,x_k)\in(\mathbb Z_n^{**})^k:
x_1^e+\cdots+x_k^e\equiv c\pmod n\}.
\end{align}
Then $\mathcal{N}_{k,c,e}(n)=|S_{k,c,e}(n)|$.
Let $n=n_1n_2$ with $n_1, n_2\in\mathbb{Z}^+$ and
$\gcd(n_1,n_2)=1$. As usual, we write
$S_{k,c,e}(n_1)\times S_{k,c,e}(n_2)$
for the Cartesian product. Then we
define the map $\theta$ as follows:
\begin{align*}
\theta:S_{k,c,e}(n)&\rightarrow S_{k,c,e}(n_1)\times S_{k,c,e}(n_2),\\
(x_1,...,x_k)&\mapsto(x_{11},...,x_{k1},x_{12},...,x_{k2}),
\end{align*}
where $x_{ji}\equiv x_j\pmod{n_i}$ and $x_{ji}\in\mathbb Z_{n_i}$
for all integers $i$ and $j$ with $1\le i\le 2$ and $1\le j\le k$.

Let us first show that $\theta$ is well defined. To do so,
we pick any $(x_1,...,x_k)\in S_{k,c,e}(n)$ with
$\theta(x_1,...,x_k)=(x_{11},...,x_{k1},x_{12},...,x_{k2})$.
Let $1\le j\le k$. As $x_j\in\mathbb{Z}_n^{**}$,
we have $\gcd(x_j,n)=\gcd(1-x_j,n)=1$. Then
$\gcd(x_{ji},n_i)=\gcd(1-x_{ji},n_i)=1$. It follows that
$x_{ji}\in\mathbb Z_{n_i}^{**}$ for $1\le i\le 2$.
Moreover, since $(x_1,...,x_k)\in S_{k,c,e}(n)$, one has
$x_1^e+\cdots+x_k^e\equiv c\pmod{n}$. Since
$x_{ji}\equiv x_j\pmod{n_i}$ and $n=n_1n_2$,
we have $$x_{1i}^e+\cdots+x_{ki}^e\equiv c\pmod{n_i}.$$
It infers that $(x_{11},...,x_{k1})\in S_{k,c,e}(n_1)$
and $(x_{12},...,x_{k2})\in S_{k,c,e}(n_2)$.
Hence $\theta$ is well defined.
So to prove that Lemma 2.2, it is enough
to prove that $\theta$ is a bijection.

On the one hand, let
$(x_1,...,x_k), (y_1,...,y_k)\in S_{k,c,e}(n)$
with $(x_1,...,x_k)\ne(y_1,...,y_k)$.
Then there exist $j\in\{1,...,k\}$
such that $x_j\not\equiv y_j\pmod n$.
It is easy to see that $x_{j1}\not\equiv y_{j1}\pmod{n_1}$
or $x_{j2}\not\equiv y_{j2}\pmod{n_2}$ as
$n=n_1n_2$ and $\gcd(n_1,n_2)=1$. Then
$\theta(x_1,...,x_k)\ne\theta(y_1,...,y_k)$.
Thus $\theta$ is an injection.

On the other hand, let $(x_{11}, ..., x_{k1}, x_{12}, ..., x_{k2})$
be any element of $ S_{k,c,e}(n_1)\times S_{k,c,e}(n_2)$.
Since $\gcd(n_1,n_2)=1$, Lemma \ref{lem1}
tells us that there exist a unique $k$-tuple
$(x_1,...,x_k)\in\mathbb{Z}_n^k$ such that
\begin{align*}
x_j\equiv x_{ji}\pmod{n_i}\ (1\le i\le 2, 1\le j\le k).
\end{align*}
Thus
$$x_1^e+\cdots+x_k^e\equiv x_{1i}^e
+\cdots+x_{ki}^e\equiv c\pmod{n_i}$$
holds for $i\in\{1,2\}$. Therefore
\begin{align*}
x_1^e+\cdots+x_k^e\equiv c\pmod{n}
\end{align*}
which implies that $(x_{1},...,x_{k})\in S_{k,c,e}(n)$.
Hence $\theta$ is a surjection, and so $\theta$ is a bijection.
It then follows immediately that $\mathcal{N}_{k,c,e}$
is a multiplicative function as desired. So
Lemma \ref{lem2} is proved.
\end{proof}

For the case $e=1$, it is easy to see that
$\mathcal{N}_{k,c,1}(p^s)=p^{(k-1)(s-1)}\mathcal{N}_{k,c,1}(p)$
since any solution $(a_1,...,a_k)\in(\mathbb{Z}_p^{**})^k$ of the
congruence $x_1+\cdots+x_k\equiv c\pmod p$ can be lifted to
$(a_1+b_1,...,a_{k-1}+b_{k-1},a_k-\sum_{i=1}^{k-1}b_i)
\in(\mathbb{Z}_{p^s}^{**})^k$ such that
$$(a_1+b_1)+\cdots+(a_{k-1}+b_{k-1})+(a_k-\sum_{i=1}^{k-1}b_i)
\equiv c\pmod p,$$
where $b_1,...b_{k-1}\in p\mathbb{Z}_{p^s}$.
However, for the higher $e$-th diagonal congruence with $e\ge 2$,
the similar argument cannot guarantee the truth of the same
relationship between $\mathcal{N}_{k,c,e}(p^s)$ and
$\mathcal{N}_{k,c,e}(p)$. In this case, we need a new tool,
i.e., the celebrated Hensel's lemma which is stated as follows.

\begin{lem} \label{lem3} \cite{[K]} (Hensel's lemma)
Let $f(x)=c_0+c_1x+c_2x^2+\cdots+c_nx^n$ be a polynomial
whose coefficients are $p$-adic integers. Let
$f'(x)=c_1+2c_2x+\cdots+nc_nx^{n-1}$ be the derivative of
$f(x)$. Let $a_0$ be a $p$-adic integer such that
$f(a_0)\equiv 0\pmod p$ and $f'(a_0)\not\equiv 0\pmod p$.
Then there exists a unique $p$-adic integer $a$ such that
$$
f(a)=0\ and\ a\equiv a_0\pmod p.
$$
\end{lem}

We can now use Hensel's lemma to show the following relation
between $\mathcal{N}_{k,c,e}(p^s)$ and $\mathcal{N}_{k,c,e}(p)$.

\begin{lem}\label{lem4}
Let $k, c, e$ and $s$ be integers with $k\ge 2$ and $e,s\ge 1$
and let $p$ be a prime number coprime to $e$. Then
$$
\mathcal{N}_{k,c,e}(p^s)=p^{(k-1)(s-1)}\mathcal{N}_{k,c,e}(p).
$$
\end{lem}
\begin{proof}
Let $S_{k,c,e}(n)$ be given as in (\ref{2.1}).
Then letting $(a_1,...,a_k)\in S_{k,c,e}(p)$ gives
\begin{align}\label{2.4}
a_1^e+\cdots+a_k^e\equiv c\pmod p
\end{align}
and
\begin{align}\label{2.5}
\gcd(a_i,p)=\gcd(1-a_i,p)=1 {\rm\ for\ }1\le i\le k.
\end{align}

Now let ${\bf b}=(b_1, ..., b_{k-1})\in(p\mathbb{Z}_{p^s})^{k-1}$,
and we define the associated function $f_{\bf b}$ as follows:
$$
f_{\bf b}(x):=(a_1+b_1)^e+\cdots+(a_{k-1}+b_{k-1})^e+x^e-c.
$$
Noticing that $b_1\equiv ...\equiv b_{k-1}\equiv 0\pmod p$,
by (\ref{2.4}) one has
\begin{align}\label{2.6}
f_{\bf b}(a_k)
&=(a_1+b_1)^e+\cdots+(a_{k-1}+b_{k-1})^e+a_k^e-c\notag\\
&\equiv a_1^e+\cdots+a_{k-1}^e+a_k^e-c\notag\\
&\equiv 0\pmod p.
\end{align}
Since $p\nmid e$, $a_k\in\mathbb{Z}_p^{**}$ and
$f_{\bf b}'(x)=ex^{e-1}$, we have
\begin{align}\label{2.7}
f_{\bf b}'(a_k)=ea_k^{e-1}\not\equiv 0\pmod p.
\end{align}
So by (\ref{2.6}) and (\ref{2.7}), Lemma \ref{lem3}
guarantees the existence of the unique $p$-adic integer
$a$ such that $f_{\bf b}(a)=0$ and $a\equiv a_k\pmod p.$
Now we write $a:=\beta_0+\beta_1p+\beta_2p^2+\cdots$
with $\beta_0,\beta_1,\beta_2,...\in\{0,...,p-1\}$,
and let $a':=\beta_0+\beta_1p+\beta_2p^2+\cdots+\beta_{s-1}p^{s-1}$.
Then $a'\equiv a\pmod {p^s}$. This implies that
$f_{\bf b}(a')\equiv f_{\bf b}(a)\equiv 0\pmod {p^s}$.
Therefore
\begin{align}\label{2.8}
(a_1+b_1)^e+\cdots+(a_{k-1}+b_{k-1})^e+a'^e\equiv c\pmod {p^s}.
\end{align}

On the other hand, since $b_1,...,b_{k-1}\in p\mathbb{Z}_{p^s}$
and $a'\equiv a\equiv a_k\pmod p$, by (\ref{2.5}) one derives that
\begin{align}\label{2.9'}
(a_1+b_1,...,a_{k-1}+b_{k-1},a')\in(\mathbb{Z}_{p^s}^{**})^k.
\end{align}
Hence for any given $(k-1)$-tuple ${\bf b}=(b_1, ..., b_{k-1})
\in(p\mathbb{Z}_{p^s})^{k-1}$, (\ref{2.8}) together with
(\ref{2.9'}) tells us that there exist a unique $k$-tuple
$(a_1+b_1,...,a_{k-1}+b_{k-1},a')\in S_{k,c,e}(p^s)$
which is congruent to $(a_1,...,a_{k-1},a_k)\in S_{k,c,e}(p)$
modulo $p$.

Since every component of each $(k-1)$-tuple
${\bf b}=(b_1, ..., b_{k-1})\in(p\mathbb{Z}_{p^s})^{k-1}$
has $|p\mathbb{Z}_{p^s}|=p^{s-1}$ choices,
we can conclude that each element
$(a_1,...,a_{k-1},a_k)$ in $S_{k,c,e}(p)$
produces exactly $p^{(s-1)(k-1)}$ different elements
$(a_1+b_1,...,a_{k-1}+b_{k-1},a')$ in $S_{k,c,e}(p^s)$.
Hence $|S_{k,c,e}(p^s)|=p^{(s-1)(k-1)}|S_{k,c,e}(p)|$.
That is,
$\mathcal{N}_{k,c,e}(p^s)=p^{(k-1)(s-1)}\mathcal{N}_{k,c,e}(p)$
as desired.

This ends the proof of Lemma \ref{lem4}.
\end{proof}

The next lemma gives an exact formula for the cardinality
of $\mathbb{Z}_n^{**}$, which can be used to calculate
$\mathcal{N}_{k,c,e}(n)$ for the case when $n$ is an even integer.

\begin{lem}\label{lem5}\cite{[HJ]}
Let $n$ be an integer with $n\ge 2$. Then
$$
|\mathbb{Z}_n^{**}|=n\prod_{p|n}\Big(1-\frac{2}{p}\Big).
$$
\end{lem}

Finally, Lemmas \ref{lem6} and \ref{lem7} are identities
corresponding to exponential sums and quadratic Gauss sums
which can be used to compute $\mathcal{N}_{k,c,2}(p)$
for any odd prime $p$.

\begin{lem}\label{lem6}\cite{[A]}
Let $p$ be an odd prime and let $\alpha$
be an integer with $p\nmid\alpha$. Then
$$
\sum_{x=0}^{p-1}e\Big(\frac{\alpha x^2}{p}\Big)
=\varepsilon_p\sqrt p\Big(\frac{\alpha}{p}\Big),
$$
where
\begin{align}\label{2.9}
\varepsilon_p
:=\left\{\begin{array}{ll}
1, & {\it if}\ p\equiv 1\pmod 4,  \\
{\rm i}, & {\it if}\ p\equiv 3\pmod 4.
\end{array}\right.
\end{align}
\end{lem}

\begin{lem}\label{lem7}\cite{[A]}
Let $p$ be an odd prime and let $\alpha$ be an
integer with $p\nmid\alpha$. Then
$$
\sum_{x=1}^{p-1}\Big(\frac{x}{p}\Big)e\Big(\frac{\alpha x}{p}\Big)
=\varepsilon_p\sqrt p\Big(\frac{\alpha}{p}\Big),
$$
where $\varepsilon_p$ is given as in (\ref{2.9}).
\end{lem}

\section{Proof of Theorem \ref{thm1}}
In this section, we give the proof of
Theorem \ref{thm1}.\\
\\
{\it Proof of Theorem \ref{thm1}.}
First of all, since $\mathcal N_{k,c,e}(1)=1$, Theorem
\ref{thm1} is true when $n=1$. In what follows, we let
$n\ge 2$. If $n$ is even, then Lemma \ref{lem4} tells that
$$
|\mathbb{Z}_n^{**}|=n\prod_{p|n}\Big(1-\frac{2}{p}\Big)=0.
$$
That is, there is no exceptional unit in $\mathbb{Z}_n$.
Hence $\mathcal{N}_{k,c,2}(n)=0$ if $2\mid n$.
In the following, we can always assume that $n$ is odd.

Consequently, let $n$ be an odd integer and
$n=\prod_{i=1}^lp_i^{s_i}$ be its standard
factorization. It then follows from Lemmas
\ref{lem2} and \ref{lem4} that
\begin{align}\label{3.1}
\mathcal{N}_{k,c,2}(n)=\prod_{i=1}^lp_i^{(k-1)(s_i-1)}\mathcal{N}_{k,c,2}(p_i).
\end{align}
So we only need to compute $\mathcal{N}_{k,c,2}(p)$ for any
odd prime $p$ with $p|n$ that will be done in what follows.

As $\mathbb{Z}_p^{**}=\{2,...,p-1\}$, we have
\begin{align}\label{3.2}
\mathcal{N}_{k,c,2}(p)
&=\frac{1}{p}\sum_{(x_1,...,x_k)\in(\mathbb{Z}_p^{**})^k}
\sum_{t\in\mathbb{Z}_p}e\Big(\frac{(x_1^2+\cdots+x_k^2-c)t}{p}\Big)\notag\\
&=\frac{1}{p}\sum_{t=0}^{p-1}\Big(\sum_{x=2}^{p-1}
e\Big(\frac{x^2t}{p}\Big)\Big)^ke\Big(\frac{-ct}{p}\Big)\notag\\
&=\frac{1}{p}\Big((p-2)^k+\sum_{t=1}^{p-1}\Big(\sum_{x=2}^{p-1}
e\Big(\frac{x^2t}{p}\Big)\Big)^ke\Big(\frac{-ct}{p}\Big)\Big).
\end{align}
For any integer $t$ with $1\le t\le p-1$, with Lemma \ref{lem6}
applied to $\sum_{x=2}^{p-1}e\Big(\frac{x^2t}{p}\Big)$, one arrives at
$$
\sum_{x=2}^{p-1}e\Big(\frac{x^2t}{p}\Big)
=\varepsilon_p\sqrt p\Big(\frac{t}{p}\Big)-e\Big(\frac{t}{p}\Big)-1,
$$
where $\varepsilon_p$ is given as in (\ref{2.9}). As usual, for any
nonnegative integers $x,y,z$ and $t$ with $x+y+z=t$, we denote by
$\binom{t}{x,y,z}$ the trinomial coefficient, that is,
$\binom{t}{x,y,z}=\frac{t!}{x!y!z!}$. It then follows that
\begin{align}\label{3.3}
\Big(\sum_{x=2}^{p-1}e\Big(\frac{x^2t}{p}\Big)\Big)^k
&=\sum_{k_1+k_2+k_3=k\atop{k_1, k_2, k_3\ge 0}}\binom{k}{k_1,k_2,k_3}
\Big(\varepsilon_p\sqrt p\Big(\frac{t}{p}\Big)\Big)^{k_1}
\Big(-e\Big(\frac{t}{p}\Big)\Big)^{k_2}(-1)^{k_3}.
\end{align}
Putting (\ref{3.3}) into (\ref{3.2}), we have
\begin{align}\label{3.4}
&\mathcal{N}_{k,c,2}(p)\notag\\
=&\frac{(p-2)^k}{p}+\frac{1}{p}\sum_{t=1}^{p-1}\sum_{k_1+k_2+k_3
=k\atop{k_1, k_2, k_3\ge 0}}(-1)^{k_2+k_3}\binom{k}{k_1,k_2,k_3}
(\varepsilon_p\sqrt p)^{k_1}\Big(\frac{t}{p}\Big)^{k_1}
e\Big(\frac{{k_2}t}{p}\Big)e\Big(\frac{-ct}{p}\Big)\notag\\
=&\frac{(p-2)^k}{p}+\frac{1}{p}\sum_{k_1+k_2+k_3
=k\atop{k_1, k_2, k_3\ge 0}}(-1)^{k_2+k_3}\binom{k}{k_1,k_2,k_3}
(\varepsilon_p\sqrt p)^{k_1}\sum_{t=1}^{p-1}
\Big(\frac{t}{p}\Big)^{k_1}e\Big(\frac{({k_2}-c)t}{p}\Big).
\end{align}
Since $\gcd(t,p)=1$ for any $t\in\{1,...,p-1\}$,
one has $\big(\frac{t}{p}\big)\in\{1,-1\}$.
Then for the inner sum in (\ref{3.4}),
we need only to consider the following two cases.

(i). $k_1$ is even. Then
\begin{align}\label{3.5}
\sum_{t=1}^{p-1}\Big(\frac{t}{p}\Big)^{k_1}e\Big(\frac{(k_2-c)t}{p}\Big)
=\sum_{t=1}^{p-1}e\Big(\frac{(k_2-c)t}{p}\Big)
=\left\{\begin{array}{ll}
p-1, & {\rm if}\ p|(k_2-c),  \\
-1, & {\rm if}\ p\nmid(k_2-c).
\end{array}\right.
\end{align}

(ii). $k_1$ is odd. Noticing that $e\Big(\frac{(k_2-c)t}{p}\Big)=1$
if $p\mid (k_2-c)$ and applying Lemma \ref{lem7} to the case
$p\nmid (k_2-c)$, one derives that
\begin{align}\label{3.6}
\sum_{t=1}^{p-1}\Big(\frac{t}{p}\Big)^{k_1}e\Big(\frac{(k_2-c)t}{p}\Big)
=\sum_{t=1}^{p-1}\Big(\frac{t}{p}\Big)e\Big(\frac{(k_2-c)t}{p}\Big)
=\left\{\begin{array}{ll}
0, & {\rm if}\ p|(k_2-c),  \\
\varepsilon_p\sqrt p\big(\frac{k_2-c}{p}\big), & {\rm if}\ p\nmid(k_2-c).
\end{array}\right.
\end{align}
Hence putting (\ref{3.5}) and (\ref{3.6}) into (\ref{3.4}), one arrives at
\begin{align*}
\mathcal{N}_{k,c,2}(p)
=&\frac{(p-2)^k}{p}+\frac{1}{p}\sum_{k_1+k_2+k_3=k,k_i\ge 0
\atop{2|k_1,p|(k_2-c)}}(-1)^{k_2+k_3}\binom{k}{k_1,k_2,k_3}
(\varepsilon_p\sqrt p)^{k_1}(p-1)\\
&+\frac{1}{p}\sum_{k_1+k_2+k_3=k,k_i\ge 0
\atop{2|k_1,p\nmid(k_2-c)}}(-1)^{k_2+k_3}\binom{k}{k_1,k_2,k_3}
(\varepsilon_p\sqrt p)^{k_1}(-1)\\
&+\frac{1}{p}\sum_{k_1+k_2+k_3=k,k_i\ge 0\atop{2\nmid k_1,
p\nmid(k_2-c)}}(-1)^{k_2+k_3}\binom{k}{k_1,k_2,k_3}
(\varepsilon_p\sqrt p)^{k_1+1}\Big(\frac{k_2-c}{p}\Big)\\
=&\frac{(p-2)^k}{p}-\frac{1}{p}\sum_{k_1+k_2+k_3=k,k_i\ge 0
\atop{2|k_1}}(-1)^{k_2+k_3}\binom{k}{k_1,k_2,k_3}
(\varepsilon_p\sqrt p)^{k_1}\\
&+\sum_{k_1+k_2+k_3=k,k_i\ge 0
\atop{2|k_1,p|(k_2-c)}}(-1)^{k_2+k_3}\binom{k}{k_1,k_2,k_3}
(\varepsilon_p\sqrt p)^{k_1}\\
&+\frac{1}{p}\sum_{k_1+k_2+k_3=k,k_i\ge 0
\atop{2\nmid k_1,p\nmid(k_2-c)}}(-1)^{k_2+k_3}\binom{k}{k_1,k_2,k_3}
(\varepsilon_p\sqrt p)^{k_1+1}\Big(\frac{k_2-c}{p}\Big)\\
:=&\frac{(p-2)^k}{p}-\frac{A}{p}+B+\frac{C}{p}.
\end{align*}

On the other hand, since
$\varepsilon_p^2=(-1)^{\frac{p-1}{2}}$
and
$$\binom{t}{x,y,z}=\binom{t}{x}\binom{t-x}{y}$$
holds for any nonnegative integers $x, y, z$ and
$t$ with $t=x+y+z$, one can easily deduce that
\begin{align*}
A=&\sum_{2k_1+k_2+k_3=k, {k_i\ge 0}}(-1)^{k_2+k_3}
\binom{k}{2k_1,k_2,k_3}(\varepsilon_p\sqrt p)^{2k_1}\\
=&(-1)^{k}\sum_{k_1=0}^{\lfloor\frac{k}{2}
\rfloor}\binom{k}{2k_1}(\varepsilon_p^{2}p)^{k_1}
\sum_{k_2=0}^{k-2k_1}\binom{k-2k_1}{k_2}\\
=&(-1)^{k}\sum_{k_1=0}^{\lfloor\frac{k}{2}\rfloor}
\binom{k}{2k_1}2^{k-2k_1}(-1)^{\frac{(p-1)k_1}{2}}p^{k_1},
\end{align*}
\begin{align*}
B=&\sum_{2k_1+k_2+k_3=k,k_i\ge 0\atop{p|(k_2-c)}}(-1)^{k_2+k_3}
\binom{k}{2k_1,k_2,k_3}(\varepsilon_p\sqrt p)^{2k_1}\\
=&(-1)^{k}\sum_{k_1=0}^{\lfloor\frac{k}{2}
\rfloor}\binom{k}{2k_1}(-1)^{\frac{(p-1)k_1}{2}}p^{k_1}
\sum_{k_2=0\atop{p|(k_2-c)}}^{k-2k_1}\binom{k-2k_1}{k_2}
\end{align*}
and
\begin{align*}
C=& \sum_{2k_1+k_2+k_3=k-1,k_i\ge 0
\atop{p\nmid(k_2-c)}}(-1)^{k_2+k_3}\binom{k}{2k_1+1,k_2,k_3}
(\varepsilon_p\sqrt p)^{2k_1+2}\Big(\frac{k_2-c}{p}\Big)\\
=& (-1)^{k-1}\sum_{k_1=0}^{\lfloor\frac{k-1}{2}\rfloor}
\binom{k}{2k_1+1}(-1)^{\frac{(p-1)(k_1+1)}{2}}p^{k_1+1}
\sum_{k_2=0\atop{p\nmid(k_2-c)}}^{k-2k_1-1}
\binom{k-2k_1-1}{k_2}\Big(\frac{k_2-c}{p}\Big).
\end{align*}
Therefore
\begin{align}\label{3.7}
&\mathcal{N}_{k,c,2}(p)\notag\\
=&\frac{(-1)^{k}}{p}\Bigg((2-p)^k-\sum_{k_1=0}^{\lfloor\frac{k}{2}\rfloor}
\binom{k}{2k_1}(-1)^{\frac{(p-1)k_1}{2}}p^{k_1}\Big(2^{k-2k_1}
-p\sum_{k_2=0\atop{p|(k_2-c)}}^{k-2k_1}\binom{k-2k_1}{k_2}\Big)\notag\\
&-\sum_{k_1=0}^{\lfloor\frac{k-1}{2}\rfloor}
\binom{k}{2k_1+1}(-1)^{\frac{(p-1)(k_1+1)}{2}}p^{k_1+1}
\sum_{k_2=0\atop{p\nmid(k_2-c)}}^{k-2k_1-1}
\binom{k-2k_1-1}{k_2}\Big(\frac{k_2-c}{p}\Big)\Bigg).
\end{align}

Finally, the expected result follows immediately from
(\ref{3.1}) and (\ref{3.7}).

This concludes the proof of Theorem \ref{thm1}. \qed

\bibliographystyle{amsplain}

\end{document}